\documentclass[11pt,a4paper]{article}
\usepackage{color}
\usepackage{graphicx}
\usepackage{amsfonts}
\usepackage{amsmath,amsthm,amssymb,color}
\usepackage{hyperref}
\usepackage{eepic}
\usepackage{lineno}
\usepackage{enumerate}	
\usepackage{paralist}
\usepackage{tikz}

\theoremstyle{theorem}
\newtheorem{theorem}{Theorem}[section]

\newtheorem{lemma}{Lemma}[section]

\theoremstyle{definition}

\newtheorem{claim}{Claim}[section]

\newtheorem{conjecture}{Conjecture}[section]
\newtheorem{question}{Question}

\begin{document}

\title{\bf Large rainbow matchings in edge-colored graphs with given average color degree
\thanks{Supported by NSFC (Nos. 11871311, 11631014).}}
\date{}
\author{Wenling Zhou \thanks{School of Mathematics,
Shandong University, Jinan 250100, P.R. China. Email:
\texttt{gracezhou@mail.sdu.edu.cn}.}
}
\maketitle

\begin{center}
\begin{minipage}{130mm}
\small\noindent{\bf Abstract:}
A rainbow matching in an edge-colored graph is a matching in which no two edges have the same color. The color degree of a vertex $v$ is the number of different colors on edges incident to $v$.
Kritschgau [\textit{Electron. J. Combin. 27(2020)}] studied the existence of rainbow matchings in edge-colored graph $G$ with average color degree at least $2k$,
and proved some sufficient conditions for a rainbow marching of size $k$ in $G$. The sufficient conditions include that $|V(G)|\ge 12k^2+4k$, or $G$ is a properly edge-colored graph with $|V(G)|\ge 8k$.

In this paper, we show that every edge-colored graph $G$ with $|V(G)|\ge 4k-4$ and average color degree at least $2k-1$ contains a rainbow matching of size $k$.
In addition, we also prove that every strongly edge-colored graph $G$ with average degree at least $2k-1$ contains a rainbow matching of size at least $k$. The bound is sharp for complete graphs.

\end{minipage}
\end{center}

\smallskip

\section{Introduction}
We use~\cite{MR3822066} for terminology and notation not defined here and only consider simple undirected graphs. An \textit{edge-colored} graph is a graph in which each edge is assigned a color. Given an edge-colored graph $G$, we call it a \textit{properly edge-colored} graph if its any two adjacent edges have different colors.
Thus, in a properly edge-colored graph, edges of the same color form
a matching. If for each color $\alpha$, the set of edges having color $\alpha$ forms an induced matching in $G$, then we say that $G$ is \textit{strongly edge-colored}. Therefore, a strongly edge-colored graph is always properly edge-colored.
Furthermore, a matching $M$ in an edge-colored graph $G$ is a \textit{rainbow matching} if the edges in $M$ have distinct colors.

Given an edge-colored graph $G=(V,E)$, we use $\delta (G)$ and $d(G)$ to denote the minimum degree and the average degree of $G$ respectively. For a vertex $v\in V$, the \textit{color degree}, $\hat d_G(v)$ of $v$ is the number of different colors on edges incident to $v$.
When it is clear from the context what $G$ is, we would omit the subscript.
We use $\hat{\delta}(G)$, $\hat{\Delta}(G)$ and $\hat d(G)$ to denote
the \textit{minimum color degree}, the \textit{maximum color degree} and the \textit{average color degree} of $G$ respectively, i.e., $\hat{\delta}(G)=\min\{\hat d(v): v\in V\}$, $\hat{\Delta}(G)=\max\{\hat d(v): v\in V\}$ and $\hat d(G)=\sum_{v\in V}{\hat d(v)}/n$.
Clearly, for an edge-colored graph $G$, we have $d(v)\geq \hat d(v)$ for each $v\in V$.

Rainbow matchings in edge-colored graphs were originally studied in connection to the famous conjecture of Ryser~\cite{Ryser}, which equivalently states that every properly edge-colored complete bipartite graph $K_{n, n}$ with $n$ colors contains a rainbow matching of size $n$, where $n$ is odd. Unlike uncolored matchings for which the maximum matching problem is solvable in polynomial time, the maximum rainbow matching problem is \textit{NP}-Complete, even for bipartite graphs,
mentioned in Garey and Johnson \cite{1979Computers} as the multiple choice matching problem. Therefore, the existence of rainbow matchings has also been studied in its own right.

During the last decades, many researchers have studied the sufficient conditions to ensure that a properly edge-colored graph $G$ has a rainbow matching of size $\delta(G)$. In~\cite{MR2831098}, Wang asked does there exist a function $f(\delta(G))$, such that every properly edge-colored graph $G$ with $|V(G)|\geq f( \delta(G))$ contains a rainbow matching of size $\delta(G)$.
Diemunsch et al.~\cite{MR2946110} proved that such function does exist and $f( \delta(G))\leq \frac{98}{23}\delta(G)$. Gy\'arf\'as and S\'ark\"{o}zy~\cite{MR3192419} improved the result to $f( \delta(G))\leq 4\delta(G)-3$.
Later, this problem was generalized to find the function $f(\hat \delta(G))$ for any edge-colored graph $G$.
Lo and Tan~\cite{MR3167016} showed that $f(\hat \delta(G))\leq 4\hat \delta(G)-4$ is sufficient for $\hat \delta(G)\geq 4$.
As far as we know, the best result so far is $f(\hat \delta(G))\leq\frac{7}{2}\hat \delta(G)+2$ in~\cite{MR3357798}.
In addition, the lower bound for the size $r(G)$ of the maximum rainbow matchings in edge-colored graph $G$ has also been studied independently, in terms of the minimum color degree of $G$.
In~\cite{MR2465762}, Li and Wang showed that $r(G)\ge \lceil \frac{5\hat \delta(G)-13}{12}\rceil$ for every edge-colored graph $G$, and they conjectured that $r(G)\ge \lceil \hat \delta(G)/2\rceil$ for $\hat \delta(G)\geq 4$.
Consider a properly edge-colored $K_4$, whose edges of the same color form a matching of size $2$. For convenience, it is denoted as $\widetilde{K_4}$.
It is easy to verify that $\widetilde{K_4}$ has no a rainbow matching of size $2$, which motivates the restriction $\hat \delta(G)\geq 4$. In particular, the bound of this conjecture is sharp for properly edge-colored complete graphs. This conjecture was partially confirmed in~\cite{2010Rainbow} and fully confirmed in~\cite{MR2900062}. In particular, Kostochka and Yancey~\cite{MR2900062} proved that if $G$ is not $\widetilde{K_4}$, then $r(G)\ge \lceil \delta(G)/2\rceil$.

Since the maximum rainbow matchings problem in edge-colored graphs in terms of the minimum color degree is well studied, it is natural to study this problem under the average color degree condition.
Michael Ferrara raised~\cite{kritschgau2020rainbow} the following related question during the Rocky Mountain and Great Plains Graduate Research Workshop in Combinatorics in 2017.

\begin{question}\label{Q1}
If $G$ is an edge-colored graph with $\hat d(G)\geq 2k$, does $G$ contain a rainbow matching of size $k$?
\end{question}

Since the average color degree condition is weaker than the minimum color degree, it is more difficult to study the maximum rainbow matchings problem under the average color degree condition. Therefore, there are few known results under the average color degree condition.
Recently, Kritschgau~\cite{kritschgau2020rainbow} studied Question~\ref{Q1}, and proved some sufficient conditions to bound from below $r(G)$ in $G$ with a prescribed average color degree.
We denote by $C_i$ the cycle with $i$ vertices.

\begin{theorem}[Kritschgau~\cite{kritschgau2020rainbow}]\label{K}
Each condition below guarantees that $r(G)\ge k$ for each edge-colored graph $G$ with $\hat d(G)\geq 2k$.
\begin{compactenum}
\item[\rm (i)] $G$ is $C_3$-free.
\item[\rm (ii)] $G$ is $C_4$-free.
\item[\rm (iii)] $G$ is properly edge-colored and $|V(G)|\geq 8k$.
\item[\rm (iv)] $|V(G)|\geq 12k^{2}+4k$.
\end{compactenum}
\end{theorem}

Though Kritschgau~\cite{kritschgau2020rainbow} did not resolve Question~\ref{Q1} for all graphs, he believe the answer is affirmative.
Recall that Kostochka and Yancey showed that $r(G)\ge k$ for all edge-colored graph $G$ with $\hat \delta(G)\ge 2k-1$ and $G\neq \widetilde{K_4}$.
We want to study the consistency of the maximum rainbow matching between the minimum color degree condition and the average color degree condition. Generalising Question~\ref{Q1}, we ask the following question.

\begin{question}\label{Q2}
If $G$ is not $\widetilde{K_4}$ and $\hat d(G)\geq 2k-1$, does $G$ contain a rainbow matching of size $k$?
\end{question}

If the answer of Question~\ref{Q2} is affirmative, it would be best possible,
because the properly edge-colored complete graph $K_{t+1}$ satisfies $\hat d(K_{t+1})\geq t$ for each $t\in \mathbb{N}$, but $r(K_{t+1})\le \lceil t/2\rceil$.
In this paper, we partially resolve Question~\ref{Q2} and obtain the following result.

\begin{theorem}\label{result1}
For all $k\in \mathbb{N}^*$, let $G$ is an edge-colored graph with $\hat d(G) \geq 2k-1$ and $G\neq \widetilde{K_4}$. If $|V(G)|\geq 4k-4$, then $r(G)\ge k$.
\end{theorem}

\noindent {\bf Remark}. Theorem~\ref{result1} implies that, for any $k$, only finitely many edge-colored
graphs with average color degree at least $2k-1$ can fail to have a rainbow matching of size $k$. Furthermore, it is easy to verify that these graphs $G$ that may fail satisfy $|E(G)|\ge |V(G)|^2/4+3|V(G)|/4$. By the Tu\'ran number of $C_3$ and $C_4$, these graphs all contain $C_3$ and $C_4$. Therefore, Theorem~\ref{result1} can deduce Theorem~\ref{K}.

In addition, the topic of rainbow matchings in strongly edge-colored graphs in terms of the minimum color degree has been also well studied. Note that for strongly edge-colored graph $G$, we have $d(v)=\hat d(v)$ for each $v\in V(G)$. In 2015 Babu-Chandran-Vaidyanathan~\cite{Babu2015} showed that $r(G)\ge \lfloor 3\delta(G)/4\rfloor$ for any strongly edge-colored graph $G$ with $|V(G)|\ge 2\lfloor 3\delta(G)/4\rfloor$.
They also proposed an interesting question: Is there a constant $c$ greater than $3/4$ such that every strongly edge-colored graph $G$ has $r(G)\ge \lfloor c\delta(G)\rfloor$ if $|V(G)|\ge  2\lfloor c \delta(G)\rfloor$? Clearly, $c\le 1$. The best result so far on this question is from Cheng-Tan-Wang~\cite{Cheng2018}, and they proved the following result.

\begin{theorem}[Cheng-Tan-Wang~\cite{Cheng2018}]\label{strongly}
If $G$ is a strongly edge-colored graph with $|V(G)|\ge 2\delta(G)+1$, then
$r(G)\ge \delta(G)$.
\end{theorem}

Rather than considering host graphs with a prescribed minimum
color degree, we consider host graphs with a prescribed average color degree, and obtain the following a sharp result.

\begin{theorem}\label{result2}
For any $k\in \mathbb{N}^*$, if $G$ is a strongly edge-colored graph with $ d(G) \geq 2k-1$, then $r(G)\ge k$.
\end{theorem}

Next, we will prove Theorem~\ref{result1} and Theorem~\ref{result2} in Section 2 and Section 3 respectively. Finally, we close the paper with some remarks and conjectures in Section 4.

\section{Proof of Theorem \ref{result1}}
In this section, we will prove Theorem~\ref{result1} by induction on $k$. The base case $k = 1$ is trivial.
Suppose $k\geq 2$, and let $G$ with edge coloring $\varphi$ be a counterexample to Theorem~\ref{result1} with the fewest edges. Let $2k-1:=\hat d(G)$ and $n:=|V(G)|$ with $n\ge 4k-4$.

For the sake of contradiction, we will study the total color degree of $G$ in the following proofs. Let us start with some useful notation.
For simplicity, set $V:=V(G)$. For $v\in V$, the \emph{neighborhood} of $v$ is denoted by let $N(v):=\{u\in V \mid uv\in E(G)\}$. For $U\subseteq V$, Let $G[U]$ denote the induced subgraph of $G$ on vertex set $U$. The color used on $G[U]$ will be denoted $\varphi(G[U])$, i.e., $\varphi(G[U])=\{\varphi(e): e\in E(G[U])\}$.
If $U=V$, then we write that $\varphi(G)$ simply.

\subsection{Preliminary results}
By induction hypothesis, $r(G)=k-1$. Choose a rainbow matching $M$ of size $k-1$ in $G$, which maximizes $|\varphi(G[V\setminus V(M)])|$.
Let $H$ be the subgraph induced by $G[V\setminus V(M)]$, and let $a=|\varphi(H)|$. Clearly, $0\leq a\leq k-1$.
We say that a color appearing in $G$ is \emph{free} if it does not appear on an edge of $M$, otherwise it is \emph{unfree}. Therefore, we can divide $\varphi(G)$ into two disjoint subset $\varphi_f$ and $\varphi_{uf}$, where $\varphi_{uf}=\{\varphi(e): e\in E(M)\}$ and $\varphi_f=\varphi(G)\setminus \varphi_{uf}$.
For every vertex $v\in V$, let $\hat{d}^{f}(v)$ and $\hat{d}^{uf}(v)$ denote the \emph{free color degree} and the \emph{unfree color degree} of $v$ in $G$ respectively. Clearly, we have $\hat{d}(v)=\hat{d}^{f}(v)+\hat{d}^{uf}(v)$.
Without loss of generality, let $E(M)=\{u_iv_i: 1\leq i\leq k-1\}$. For $1\le i \le k-1$, let $B_i$ denote the bipartite subgraph of $G$ whose edge connect $\{u_i,v_i\}$ with $V(H)$, and set $h(i):=\sum_{w\in V(H)} \hat{d}^{f}_{B_i} (w)$.
This notation $h(i)$ will be used for the rest of the paper.
First, we find a property of $G$.

\begin{claim}\label{maxdegree}
$\hat \Delta(G)\leq 2(k-1)+a$.
\end{claim}

\begin{proof}
Assume for the sake of contradiction that there is a vertex $v\in V$ with $\hat d(v)\geq 2k+a-1$. Let $G^{*}=G-v$, which is obtained from $G$ by deleting the vertex $v$ and all edges incident with $v$. Since $\hat d(v)\le n-1$ and $\hat{d}(G)=2k-1$, we have
\[
\hat{d}(G^*)\geq \frac{(2k-1)n-2(n-1)}{n-1}> 2(k-1)-1.
\]
By induction hypothesis, $G^{*}$ contains a rainbow matching $M^*$ of size $k-1$. Let $H^*$ be the induced subgraph of $G$ on vertex set $V\setminus V(M^*)$, we have $|\varphi(H^*)|\leq |\varphi(H)|= a$. Since $\hat d(v) \geq 2k+a-1$, there is at least one vertex $u\in N(v)$ such that $u \notin V(M^*)$ and $\varphi(uv) \notin \varphi(M^*)$. Let $M'=M^* \cup \{uv\}$, which yields a rainbow matching of size $k$ in $G$, a contradiction.
\end{proof}

Furthermore, recalling the result of Kostochka and Yancey, it follows that $\hat \delta(G)<2k-1$, otherwise $G$ contains a rainbow matching of size $k$, a contradiction. Therefore, we have $2k-1=\hat d(G)< \hat \Delta(G)\leq 2(k-1)+a$, i.e., $a\geq 2$. In addition, since $a\le k-1$, we have $k\ge3$ in the next proof.

By the minimality of $G$, it is easy to prove the following property.

\begin{lemma}[\cite{MR2900062}]\label{star}
The edges of each color class of $\varphi$ form a forest of stars.
\end{lemma}

Now, let us consider the relationship between the rainbow matching $M$ and the induced subgraph $H$.
Given a color $\alpha \in \varphi(H)$, let $H^\alpha$ denote the subgraph of $H$ with the edges in color class $\alpha$, and $s_H(\alpha)$ denote the number of stars in $H^\alpha$. Since $\varphi(H)\subseteq \varphi_{uf}$, we partition $M$ into $X_1,X_2,X_3$ as following:
\begin{compactenum}
  \item [\rm (1)]For every  $e\in X_{1}$, $s_H(\varphi(e))\geq 2$;
  \item [\rm (2)]For every $e\in X_{2}$, $s_H(\varphi(e))=1$;
  \item [\rm (3)]For every $e\in X_{3}$, $s_H(\varphi(e))=0$.
\end{compactenum}

A \emph{free} edge in $G$ is an edge colored with a free color. For $v\in V(M)$, let $E_H^{f}(v)$ denote the set of free edges connecting $v$ and $V(H)$ in $G$.
In order to get a more detailed estimate, we partition $X_3$ into $Y_1,Y_2,Y_3$ as following:
\begin{compactenum}
\item[\rm (i)] For every $e\in Y_{1}$, every endpoint $v$ of edge $e$ with $|\varphi(E_H^{f}(v))|\geq 1$;
\item[\rm (ii)] For every $e\in Y_{2}$, there is only one endpoint $v$ of edge $e$ with $|\varphi(E_H^{f}(v))|\geq 1$;
\item[\rm (iii)] For every $e\in Y_{3}$, every endpoint $v$ of edge $e$ with $|\varphi(E_H^{f}(v))|=0$.
\end{compactenum}

For convenience, let $x_{j}=|X_{j}|$ and $y_{j}=|Y_{j}|$ for $1\le j\le 3$, and let $\varphi(u_iv_i)=i$ for every edge $u_iv_i\in E(M)$. Next, for the partition above, we will state and prove several claims that are useful for the proof of Theorem~\ref{result1}.
\begin{claim}\label{x1}
For every edge $u_iv_i\in X_{1}$, we have $h(i)=0$.
\end{claim}

In Figure~1 we list three configurations (1.1), (1.2) and (1.3), which can not appear in $G$, otherwise, they would yield a rainbow matching of size $k$ in $G$. Configurations (1.1) directly proves Claim~\ref{x1}.
\begin{figure}
\begin{center}
\begin{tikzpicture}
\draw [densely dashed] (-3.1,-0.7) rectangle (-1.9,1);
\draw [densely dashed] (-1,0.1) ellipse (14pt and 25pt);

\filldraw (-2.9,0.5) circle (1pt);\node at(-2.9,0.65) {$u_i$};
\filldraw (-2.1,0.3) circle (1pt);\node at(-2.1,0.45) {$v_i$};
\draw (-2.9,0.5)--(-2.1,0.3);

\filldraw (-1.3,0) circle (1pt);\filldraw (-0.7,-0.3) circle (1pt);
\draw (-1.3,0)--(-0.7,-0.3);
\node at (-1,0) {$i$};

\filldraw (-1.2,0.3) circle (1pt);
\draw [red](-2.1,0.3)--(-1.2,0.3);
\node [red]at (-1.7,0.4) {$\alpha$};
\node at (-2.5,-1) {$M$};
\node at (-1,-1) {$H$};
\node at (-1.8,-1.2) {(1.1)};

\draw[shift ={(4 ,0)}][densely dashed] (-3.1,-0.7) rectangle (-1.9,1);
\draw[shift ={(4 ,0)}] [densely dashed] (-1,0.1) ellipse (14pt and 25pt);

\filldraw (1.1,0.5) circle (1pt);\node at(1.1,0.65) {$u_i$};
\filldraw (1.9,0.3) circle (1pt);\node at(1.9,0.45) {$v_i$};
\draw (1.1,0.5)--(1.9,0.3);
\filldraw (1.1,-0.2) circle (1pt);\node at(1.1,-0.4) {$u_j$};
\filldraw (1.9,-0.4) circle (1pt);\node at(1.9,-0.6) {$v_j$};
\draw (1.1,-0.2)--(1.9,-0.4);
\draw [red](1.1,0.5)--(1.1,-0.2);
\node [red]at (1.3,0.15) {$\alpha$};

\filldraw (2.7,0.4) circle (1pt);\filldraw (3.3,0.1) circle (1pt);
\draw (2.7,0.4)--(3.3,0.1);
\node at (3,0.4) {$i$};
\filldraw (2.7,-0.2) circle (1pt);\filldraw (3.3,-0.5) circle (1pt);
\draw (2.7,-0.2)--(3.3,-0.5);
\node at (3,-0.1) {$j$};

\node at (1.5,-1) {$M$};
\node at (3,-1) {$H$};
\node at (2.2,-1.2) {(1.2)};
\draw[shift ={(8 ,0)}][densely dashed] (-3.1,-0.8) rectangle (-1.9,1);
\draw[shift ={(8 ,0)}] [densely dashed] (-1,0.1) ellipse (14pt and 25pt);

\filldraw (5.1,0.5) circle (1pt);\node at(5.1,0.65) {$u_i$};
\filldraw (5.9,0.3) circle (1pt);\node at(5.9,0.45) {$v_i$};
\draw (5.1,0.5)--(5.9,0.3);
\filldraw (5.1,-0.2) circle (1pt);\node at(5.1,-0.4) {$u_j$};
\filldraw (5.9,-0.4) circle (1pt);\node at(5.9,-0.6) {$v_j$};
\draw (5.1,-0.2)--(5.9,-0.4);
\draw [red](5.1,0.5)--(5.1,-0.2);
\node [red]at (5.3,0.15) {$\alpha$};

\filldraw (6.7,0.4) circle (1pt);\filldraw (7.3,0.1) circle (1pt);
\filldraw (6.7,-0.3) circle (1pt);
\draw (6.7,0.4)--(7.3,0.1);\draw (6.7,-0.3)--(7.3,0.1);
\node at (7,0.4) {$i$};\node at (7,-0.4)  {$j$};

\draw [blue](5.9,0.3)--(6.7,0.4);
\node [blue]at (6.3,0.5) {$\beta$};

\node at (5.5,-1) {$M$};
\node at (7,-1) {$H$};
\node at (6.2,-1.2) {(1.3)};
\end{tikzpicture}
\caption{Some configurations that can not appear in $G$.}
\end{center}
\end{figure}
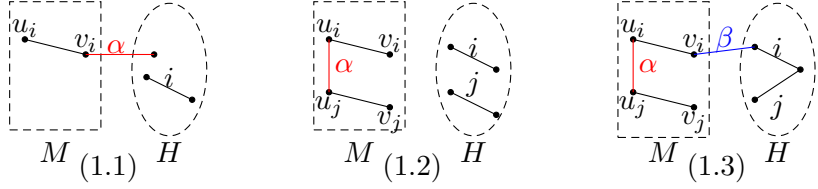

\begin{claim}\label{x2}
For every edge $u_iv_i\in X_{2}$, if $|E(H^i)|=1$, then we have $h(i)\le 4$; if $|E(H^i)|\ge 2$, then we have $h(i)\le 2$.
\end{claim}
In Figure~2 we list the extremal configurations for $|E(H^i)|=1$ with $h(i)= 4$ (see Figure~(2.1)), and $|E(H^i)|\ge 2$ with $h(i)= 2$ (see Figure~(2.2)). In particular, for the extremal configurations (2.1), we have the following claim. We denote by $G_{3}$ the induced subgraph of $G$ on $V\setminus V(X_3)$.
\begin{figure}
\begin{center}
\begin{tikzpicture}
\draw [densely dashed] (-3.1,-0.7) rectangle (-1.9,0.8);
\draw [densely dashed] (-1,0.1) ellipse (14pt and 23pt);

\filldraw (-2.5,0.5) circle (1pt);\node at(-2.7,0.6) {$u_i$};
\filldraw (-2.5,-0.4) circle (1pt);\node at(-2.7,-0.3) {$v_i$};
\draw (-2.5,0.5)--(-2.5,-0.4);

\filldraw (-1,0.5) circle (1pt);\filldraw (-1,-0.4) circle (1pt);
\draw (-1,0.5)--(-1,-0.4);
\node at (-0.9,0) {$i$};

\draw [red](-2.5,0.5)--(-1,0.5);
\draw [red](-2.5,-0.4)--(-1,-0.4);
\node [red]at (-1.7,0.6) {$\alpha$};\node [red]at (-1.7,-0.5) {$\alpha$};
\draw [blue](-2.5,0.5)--(-1,-0.4);
\draw [blue](-2.5,-0.4)--(-1,0.5);
\node [blue]at (-2.2,0.1) {$\beta$};\node [blue]at (-1.3,0.1) {$\beta$};
\node at (-2.5,-1) {$M$};
\node at (-1,-1) {$H$};
\node at (-1.8,-1.2) {(2.1)};

\draw[shift ={(4 ,0)}][densely dashed] (-3.1,-0.7) rectangle (-1.9,0.8);
\draw[shift ={(4 ,0)}] [densely dashed](-1,0.1) ellipse (14pt and 23pt);

\filldraw (1.5,0.5) circle (1pt);\node at(1.3,0.6) {$u_i$};
\filldraw (1.5,-0.4) circle (1pt);\node at(1.3,-0.3) {$v_i$};
\draw (1.5,0.5)--(1.5,-0.4);

\filldraw (2.8,0) circle (1pt);\filldraw (3.3,0.5) circle (1pt);\filldraw (3.3,-0.4) circle (1pt);
\draw (2.8,0)--(3.3,0.5);\draw (2.8,0)--(3.3,-0.4);
\node at (3,0.4) {$i$};
\node at (3.2,-0.1) {$i$};
\draw [red](1.5,0.5)--(2.8,0);\draw [blue](1.5,-0.4)--(2.8,0);
\node [red]at (1.9,0.5) {$\alpha$};\node [blue]at (1.9,-0.5) {$\beta$};

\node at (1.5,-1) {$M$};
\node at (3,-1) {$H$};
\node at (2.2,-1.2) {(2.2)};
\end{tikzpicture}
\caption{Two extremal configurations of Claim~\ref{x2}.}
\end{center}
\end{figure}
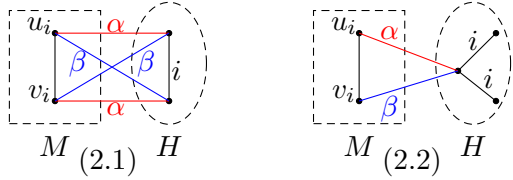

\begin{claim}\label{x2+}
For every edge $u_iv_i\in X_{2}$, if $h(i)= 4$, then $\hat d^f_{G_3}(u_i)+\hat d^f_{G_3}(v_i)= 4$.
\end{claim}

\begin{proof} Fix $u_iv_i\in X_{1}$. Since $h(i)= 4$, we have $|E(H^i)|=1$, and the extremal configurations (2.1) appears in $G$. Without loss of generality, let $\varphi(E_H^{f}(u_i))=\varphi(E_H^{f}(v_i))=\{\alpha,\beta\}$.
For any $u_jv_j\in X_1\cup X_2$ with $u_jv_j\neq u_iv_i$, since $E(H^j)\neq \emptyset$, there are $e_1\in E(H^i)$ and $e_2\in E(H^j)$ such that $e_1$ and $e_2$ either intersect or disjoint.
If $e_1\cap e_2=\emptyset$, then $\varphi(E_{u_jv_j}^{f}(u_i))=\varphi(E_{u_jv_j}^{f}(v_i))=\emptyset$, since the configurations (1.2) in Figure~1 can not appear in $G$.
If $e_1\cap e_2\neq\emptyset$, then $\varphi(E_{u_jv_j}^{f}(u_i)), \varphi(E_{u_jv_j}^{f}(v_i))\in \{\alpha,\beta\}$, since the configurations (1.3) in Figure~1 can not appear in $G$.
Therefore, $u_i$ and $v_i$ can only connect $\alpha$ color edges or $\beta$ color edges in $G_3$, i.e., $\hat d^f_{G_3}(u_i)+\hat d^f_{G_3}(v_i)= 4$.
\end{proof}

\begin{claim}\label{y_1}
For every edge $u_iv_i\in Y_1$, we have $
\hat d(u_i)+\hat d(v_i)+h(i)\leq n+2k+2a-2$.
\end{claim}

\begin{proof}
For every edge $u_iv_i\in Y_1$, we have $|\varphi(E_H^{f}(u_i))|\geq 1$ and $|\varphi(E_H^{f}(v_i))|\geq 1$. First, we claim that $|\varphi(E_H^{f}(u_i))|\leq 2$ and $|\varphi(E_H^{f}(v_i))|\leq 2$. Otherwise, without loss of generality, suppose that there are at least three vertices $w_1,w_2,w_3\in N(u_i)\cap V(H)$ such that $\varphi(u_iw_1)=\alpha_1$, $\varphi(u_iw_2)=\alpha_2$ and $\varphi(u_iw_3)=\alpha_3$, where $\{\alpha_1,\alpha_2,\alpha_3\}\subseteq \varphi_f$. Since $|\varphi(E_H^{f}(v_i))|\geq 1$, there is at least one vertex $w_k\in N(v_i)\cap V(H)$ such that $\varphi(v_iw_k)=\alpha_k$ and $\alpha_k\in \varphi_f$. In this case, we can always find two disjoint edges with different free colors to replace $u_iv_i$ and obtain a larger matching in $G$, a contradiction.
Namely, for any edge $u_iv_i\in Y_1$, $B_i$ can not contain two disjoint edges with different free colors. Next, we will discuss all the possible cases.

\textbf{Case 1: $|\varphi(E_H^{f}(u_i))|=|\varphi(E_H^{f}(v_i))|=2$.}
In this case, the set of edges with free colors in $B_i$ can only form a $\widetilde{K_4}$, whose edges of the same color form a matching of size $2$.
Otherwise, $B_i$ contains two disjoint edges with different free colors, a contradiction. Thus $h(i)=4$. By Claim \ref{maxdegree}, we have $\hat d(u_i)+\hat d(v_i)+h(i)\leq 4k+2a$.

\textbf{Case 2: $|\varphi(E_H^{f}(u_i))|=2$ \text{and} $|\varphi(E_H^{f}(v_i))=1$.} Under the circumstances, there are two vertices $w_1,w_2\in N(u_i)\cap V(H)$ such that $\varphi(u_iw_1)=\alpha_1$, $\varphi(u_iw_2)=\alpha_2$, and $\{\alpha_1,\alpha_2\}\subseteq \varphi_f$. Similarly, there is also one vertex $w_k\in N(v_i)\cap V(H)$ such that $\varphi(v_iw_k)=\alpha_k$ and $\alpha_k\in \varphi_f$. In order to avoid two disjoint edges with different free colors in $B_i$, either $w_k=w_1$ and $\alpha_k=\alpha_2$ or $w_k=w_2$ and $\alpha_k=\alpha_1$, which implies that $|E_H^{f}(v_i)|=1$.
For $w\in V(H)\backslash\{w_1,w_2\}$, if there exists $u_iw\in E(G)$, then either $\varphi(u_iw)=\alpha_k$ or $\varphi(u_iw)\in \varphi_{uf}$. Hence, $h(i)=|E_H^{f}(u_i)|+|E_H^{f}(v_i)|\le |V(H)|-|E_H^{uf}(u_i)|+1$.
In addition, recall that $\hat d(u_i)= \hat d_M(u_i)+\hat d^{f}_H(u_i)+\hat d^{uf}_H(u_i)$, $\hat d^{f}_H(u_i)=|\varphi(E_H^{f}(u_i))|=2$ and $\hat d^{uf}_H(u_i)\le |E_H^{uf}(u_i)|$. Therefore, we have
\[
\hat d(u_i)+\hat d(v_i)+h(i)\leq 2(k-1)-1+2+\hat d(v_i)+n-2(k-1)+1\le n+2k+a,
\]
where the last inequality follows from Claim \ref{maxdegree}, $\Delta^{c}(G) \leq 2(k-1)+a$.


\textbf{Case 3: $\varphi(E_H^{f}(u_i))=\{\alpha_1\}$, $\varphi(E_H^{f}(v_i))=\{\alpha_2\}$ \text{and} $\alpha_1\neq \alpha_2$.} It is easy to check $|E_H^{f}(u_i)|=|E_H^{f}(v_i)|=1$, otherwise, $B_i$ contains two disjoint edges with different free colors, a contradiction. Therefore, $\hat d(u_i)+\hat d(v_i)+h(i)\leq 2\hat \Delta(G)+2\leq 4k+2a-2$.

\textbf{Case 4: $\varphi(E_H^{f}(u_i))=\varphi(E_H^{f}(v_i))=\{\alpha\}$.} In this case, we have $\hat d(u_i)+\hat d(v_i)+h(i)\leq 2\hat \Delta(G)+|V(H)|\leq n+2k+2a-2$.

In conclusion, since $a\geq2$ and $n\geq 4k-4$, we have $\hat d(u_i)+\hat d(v_i)+h(i)\leq n+2k+2a-2$.
\end{proof}

Finally, by the definition of $Y_2$, it is easy to get the following claim.
\begin{claim}\label{y_2}
For every edge $u_iv_i\in Y_2$, we have $h(i)\leq n-2k+2$.
\end{claim}

\subsection{Estimating the total color degree of $G$}
In this section, using above claims, we will estimate the total color degree of $G$. First, recalling the definition of $h(i)$, we have
\begin{equation}\label{total}
  \begin{split}
  \sum_{v\in V}{\hat d(v)}
  = & \sum_{v\in V(M)}{\hat d(v)}+\sum_{w\in V(H)}{\hat d(w)}\\
  \leq & \sum_{u_iv_i\in E(M)}{(\hat d(u_i)+\hat d(v_i)+h(i))}+\sum_{w\in V(H)}{\hat d^{uf}(w)}\\
  \end{split}
\end{equation}
Note that $\hat d^{uf}(v)\le k-1$ for all $v\in V$, since $|\varphi_{uf}|=k-1$.
Next, we break the proof step into two cases.
\smallskip

\noindent \textbf{Case 1: $a=k-1$.}

By the definition of $a$, one should notice that $a = k-1$, which means that $M = X_1\cup X_2$.
By Claim~\ref{maxdegree}, $\hat \Delta(G)\le 3k-3$ under this case.
Recalling Claim~\ref{x1}, we have $\hat d(u_i)+\hat d(v_i)+h(i)\le 2\hat \Delta(G)\le 6(k-1)$ for any $u_iv_i\in X_1$.
By Claim \ref{x2} and \ref{x2+}, we have
\[
\hat d(u_i)+\hat d(v_i)+h(i)\le \max\{\hat d^{uf}(u_i)+\hat d^{uf}(v_i)+4+4,2\hat \Delta(G)+3\}\le 6(k-1)+3
\]
for any $u_iv_i\in X_2$, where the last inequality follows from $k\ge 3$.

According to Inequations (\ref{total}), $\hat d(G)=2k-1$ and $n\ge 4k-4$, we have
\begin{equation*}
   \begin{split}
(2k-1)n=\sum_{v\in V}{\hat d(v)}
    \leq &(k-1)(6(k-1)+3)+(n-2k+2)(k-1)\\
      < &(2k-1)n, \\
 \end{split}
 \end{equation*}
which is contradictory.
\smallskip

\noindent \textbf{Case 1: $a<k-1$.}

According to Inequations (\ref{total}), Claim \ref{maxdegree}--\ref{y_2} and $n\ge 4k-4$, we have
\begin{equation*}
   \begin{split}
\sum_{v\in V}{\hat d(v)}
    \leq & 2\hat \Delta(G)\cdot(k-1-y_1)+4x_2+y_1(n+2k+2a-2)+y_2(n-2k+2)+(n-2k+2)(k-1)\\
    \leq & 2(2k-2+a)(k-1)+4x_2+(y_1+y_2)(n-2k+2)+(n-2k+2)(k-1)\\
      < &(2k-1)n, \\
 \end{split}
 \end{equation*}
 which is also contradictory.

\section{Proof of Theorem \ref{result2}}
In this section, we will prove Theorem~\ref{result2} by induction on $k$. The base case $k = 1$ is trivial.
Suppose $k\geq 2$, and let $G$ with strongly edge coloring $\varphi$ be a counterexample to Theorem~\ref{result2} with the fewest edges. Let $2k-1:=d(G)$ and $n:=|V(G)|$. By the result of Kostochka and Yancey~\cite{MR2900062} and Theorem~\ref{strongly}, we may assume that $n\ge 2k+1$ and $\delta(G)\le k-1$.

For the sake of contradiction, we still consider the total degree of $G$ in the following proofs. Since $\delta(G)\le k-1$, there is a vertex $v\in V(G)$ such that $d(v)\le k-1$. By the minimality of $G$, we have $d(v)\ge 1$. Let $u\in N(v)$, $\varphi(uv)=\alpha$, and $G^\alpha$ denote the subgraph of $G$ with the edges in color class $\alpha$. Since $G$ is strongly edge-colored, $G^\alpha$ is an induced matching in $G$. Hence, $d(u)$ is at most $n-2|E(G^\alpha)|+1$.
Let $G^{*}$ be obtained from $G$ by deleting the vertex $v,u$ and all edges in $G^\alpha$, then $r(G^{*})< k-1$. By induction hypothesis, $d(G^{*})< 2k-3$. Therefore, we have
\begin{equation*}
   \begin{split}
(2k-1)n=\sum_{v\in V}{ d(v)}
    =&2d(u)+2d(v)+2(|E(G^\alpha)|-1)+(n-2)d(G^{*})\\
      < & 2(k-1)+2(n-2|E(G^\alpha)|+1)+2(|E(G^\alpha)|-1)+(n-2)(2k-3)\\
      < &(2k-1)n, \\
 \end{split}
 \end{equation*}
which is contradictory.

\section{Concluding Remarks}\label{Sec-re.}
Though we were not able to resolve Question~\ref{Q2} for all graphs, we believe the answer is
affirmative:
\begin{conjecture}\label{conj.}
All but $\widetilde{K_4}$ edge-colored graphs $G$ with $\hat d(G)\geq 2k-1$ contain a rainbow matching
of size at least $k$.
\end{conjecture}

We remark that using the ideas introduced in the proof of Theorem~\ref{result1}, for properly edge-colored graph $G$, it is conceivable that the lower bound for $|V(G)|$ in Theorem~\ref{result1} may be further improved. However, it would be very interesting (and seems to be difficult) to prove conjecture~\ref{conj.} for all properly edge-colored graphs. If conjecture~\ref{conj.} for all properly edge-colored graphs is ture, then it would yield a good upper bound on the rainbow Tur\'an number of matchings. Given a graph $H$, the \emph{rainbow Tur\'an number} of $H$ is defined as the maximum number of edges in a properly edge-colored graph on $n$ vertices with no rainbow copy of $H$. The systematic study of rainbow Tur\'an number was initiated in 2007 by Keevash-Mubayi-Sudakov-Verstra\"{e}te~\cite{2007Rainbow}. They asymptotically determined the rainbow Tur\'an number for any non-bipartite graph, but for the rainbow Tur\'an number of matchings, there are still no good results so far.

\section*{Acknowledgements}
The author would like to thank Professor Guanghui Wang's feedback and guidance while working on these problems. Furthermore, the author would like to thank Yangyang Cheng, Yulin Chang, Tong Li and Hao Lin's help. 

\bibliographystyle{abbrv}
\bibliography{ref}

%

\end{document}